\documentclass[11pt, a4paper]{amsart}
\usepackage[T1]{fontenc}
\usepackage[utf8]{inputenc}
\usepackage{geometry}                
\usepackage{graphicx}
\usepackage{amsmath, amsthm, amssymb, amscd}
\usepackage{mathrsfs}
\usepackage{hyperref}
\usepackage[all]{xy}
\usepackage{tikz}
\usepackage{stmaryrd}

\theoremstyle{plain}
\newtheorem{thm}{Theorem}[section]
\newtheorem{lem}[thm]{Lemma}

\theoremstyle{definition}

\theoremstyle{remark}

\newcommand{\gl}{\mathrm{GL}}

\newcommand{\bbg}{\mathbb{G}}

\newcommand{\co}{\mathcal{O}}

\newcommand{\chh}{\mathcal{H}}

\newcommand{\kg}{\mathfrak{g}}

\newcommand{\kt}{\mathfrak{t}}

\newcommand{\bn}{\mathbf{N}}
\newcommand{\bz}{\mathbf{Z}}

\newcommand{\bq}{\mathbf{Q}}

\newcommand{\bc}{\mathbf{C}}

\newcommand{\lie}{\mathrm{Lie}}

\newcommand{\prim}{\mathrm{prim}}
\newcommand{\gen}{\mathrm{gen}}
\newcommand{\reg}{\mathrm{reg}}

\author{Zongbin \textsc{Chen}}

\address{EPFL SB Mathgeom/Geom, MA B1 447\\
Lausanne, CH-1015, Switzerland} 
\email{zongbin.chen@epfl.ch} 
\title{Geometric construction of generators of CoHA of doubled quiver}

\begin{document}
\maketitle

\begin{abstract}

Let $Q$ be the double of a quiver. According to Efimov, Kontsevich and Soibelman, the cohomological Hall algebra (CoHA) associated to $Q$ is a free super-commutative algebra. In this short note, we confirm a conjecture of Hausel which gives a geometric realisation of the generators of the CoHA.

\end{abstract}

\section{Introduction}

Let $Q=(I, \Omega)$ be a quiver with a set of vertices $I$ and with $a_{ij}$ arrows from $i\in I$ to $j\in I$. For each dimension vector $\gamma=(\gamma^{i})_{i\in I}\in \bz_{\geq 0}^{I}$, we have the affine $\bc$-variety $M_{\gamma}$ of representations of $Q$ in complex coordinate space $\bigoplus_{i\in I}\bc^{\gamma^{i}}$. It is acted on by the complex algebraic group $G_{\gamma}=\prod_{i\in I}\gl_{\gamma^{i}}(\bc)$ and the action factors through $PG_{\gamma}:=G_{\gamma}/\bbg_{m}$, where $\bbg_{m}$ is embedded diagonally in $G_{\gamma}$. We denote by $[M_{\gamma}/G_{\gamma}]$ the moduli stack of representations of $Q$ of dimension $\gamma$. 

For $\gamma_1,\,\gamma_{2}\in \bz_{\geq 0}^{I}$, let $M_{\gamma_{1},\gamma_{2}}$ be the subvariety of $ M_{\gamma_{1}+\gamma_{2}}$ consisting of the representations of $Q$ such that the subspace $\bigoplus_{i\in I}\bc^{\gamma_{1}^{i}}\subset \bigoplus_{i\in I}(\bc^{\gamma_{1}^{i}}\oplus \bc^{\gamma_{2}^{i}})$ forms a sub representation. Let $G_{\gamma_{1},\gamma_{2}}\subset G_{\gamma_{1}+\gamma_{2}}$ be the subgroup preserving $\bigoplus_{i\in I}\bc^{\gamma_{1}^{i}}$. Then $[M_{\gamma_{1},\gamma_{2}}/G_{\gamma_{1},\gamma_{2}}]$ is the moduli stack classifying the extensions of representations of $Q$ of dimension vector $\gamma_{2}$ by that of dimension vector $\gamma_{1}$. We have the correspondence
$$
\xymatrix{
& [M_{\gamma_{1},\gamma_{2}}/G_{\gamma_{1},\gamma_{2}}] \ar[dl]_{p}\ar[dr]^{q}&\\
[M_{\gamma_{1}}/G_{\gamma_{1}}] \times [M_{\gamma_{2}}/G_{\gamma_{2}}] && [M_{\gamma_{1}+\gamma_{2}}/G_{\gamma_{1}+\gamma_{2}}],
}
$$
from which Kontsevich and Soibelman \cite{ks} have constructed an associative algebra structure on 
$$
\chh:=\bigoplus_{\gamma\in \bz_{\geq 0}^{I}}\chh_{\gamma},\quad \chh_{\gamma}:=H^{*}([M_{\gamma}/G_{\gamma}])=H_{G_{\gamma}}^{*}(M_{\gamma}),
$$
which is called the \emph{cohomological Hall algebra} (CoHA) of the quiver $Q$. Here and after, all the cohomological groups take coefficients in $\bq$. The resulting product has a shift in cohomological degree:
$$
H_{G_{\gamma_{1}}}^{*}(M_{\gamma_{1}})\times H_{G_{\gamma_{2}}}^{*}(M_{\gamma_{2}})\to H_{G_{\gamma_{1}+\gamma_{2}}}^{*-2\chi_{Q}(\gamma_{1},\gamma_{2})}(M_{\gamma_{1}+\gamma_{2}}), 
$$
where 
$$
\chi_{Q}(\gamma_{1},\gamma_{2})=\sum_{i\in I}\gamma_{1}^{i}\gamma_{2}^{i}-\sum_{i,j\in I}a_{i,j}\gamma_{1}^{i}\gamma_{2}^{j}.
$$

Suppose that the quiver $Q$ is symmetric, i.e. $a_{ij}=a_{ji}$. In this case, $\chh$ has more structures. First of all, one can make $\chh$ into a $(\bz^{I}_{\geq 0},\,\bz)$-graded algebra, by requiring elements in $H_{G_{\gamma}}^{k}(M_{\gamma})$ to be of bidegree $(\gamma, \,k+\chi_{Q}(\gamma,\gamma))$. Secondly, following Efimov \cite{e}, we can twist the multiplication by a sign such that $(\chh, \,*)$ is a super-commutative algebra with respect to the $\bz$-grading. In fact, for $a_{\gamma,k}\in \chh_{\gamma,k},\,a_{\gamma',k'}\in \chh_{\gamma',k'}$, we have
$$
a_{\gamma,k}a_{\gamma',k'}=(-1)^{\chi_{Q}(\gamma,\gamma')}a_{\gamma',k'}a_{\gamma,k}.
$$
We can find a bilinear form $\psi:(\bz/2)^{I}\times (\bz/2)^{I}\to \bz/2$ such that
$$
\psi(\gamma_{1},\gamma_{2})+\psi(\gamma_{2},\gamma_{1})\equiv \chi_{Q}(\gamma_{1},\gamma_{2})+\chi_{Q}(\gamma_{1},\gamma_{1})\chi_{Q}(\gamma_{2},\gamma_{2})\mod 2.
$$
Then the twisted product on $\chh$ is defined to be
$$
a_{\gamma,k}*a_{\gamma',k'}=(-1)^{\psi(\bar{\gamma},\bar{\gamma}')}a_{\gamma,k}\cdot a_{\gamma',k'},
$$
where $\bar{\gamma}$ is the image of $\gamma$ in $(\bz/2)^{I}$.

For the symmetric quiver $Q$, it is conjectured by Kontsevich and Soibelman \cite{ks} and proved by Efimov \cite{e}, that the $(\bz^{I}_{\geq 0}, \bz)$-graded algebra $(\chh, \,*)$ is a free super-commutative algebra generated by a $(\bz^{I}_{\geq 0}, \bz)$-graded vector space $V$ of the form $V=V^{\prim}\otimes \bq[x]$, with $x$ an element of degree $(0,2)$, and for all $\gamma\in \bz^{I}_{\geq 0}$ the vector space $V_{\gamma,k}^{\prim}$ is non-zero for only finitely many $k$. Geometrically, via the isogeny $G_{\gamma}\to PG_{\gamma}\times\bbg_{m}$, we have
$$
\chh_{\gamma}=H^{*}_{G_{\gamma}}(M_{\gamma})\cong H^{*}_{PG_{\gamma}}(M_{\gamma})\otimes H^{*}_{\bbg_{m}}(\mathrm{pt}),
$$
and it gives the above factorisation $V=V^{\prim}\otimes \bq[x]$. Let $M_{\gamma}^{st}$ be the stable part of $M_{\gamma}$ with respect to the $0$-stability condition, i.e. $M_{\gamma}^{st}$ consists of all the simple $Q$-modules. Then as we will see in theorem \ref{stablegenerater}, $V_{\gamma,*+\chi_{Q}(\gamma,\gamma)}^{\prim}$ is in fact the pure part of $H^{*}([M_{\gamma}^{st}/PG_{\gamma}])$, where the word ``pure'' refers to the mixed Hodge structure on the cohomological group. Let $c_{\gamma,k}=\dim_{\bc}V_{\gamma,k}^{\prim}$, the above result implies that the quantum Donaldson-Thomas invariants of the quiver $Q$ without potential and with $0$-stability condition is
$$
\Omega(\gamma)(q)=\sum_{k\in \bz}c_{\gamma, k}q^{k/2}\in \bz[q^{\pm \frac{1}{2}}].
$$
In particular, the coefficients are positive. 

On the other hand, in the work of Hausel, Letellier and Rodriquez-Villegas \cite{positive}, they found another expression for the quantum Donaldson-Thomas invariants. From now on, we work with quivers $Q=(I,\Omega)$ that are the double of another quiver, i.e. $\Omega=\Omega_{0}\sqcup \Omega_{0}^{\mathrm{op}}$, where $\Omega_{0}^{\mathrm{op}}$ is obtained by reversing all the arrows in $\Omega_{0}$. In this case, $M_{\gamma}$ is endowed with a $G_{\gamma}$-invariant holomorphic symplectic form $\omega$. Let $\mu:M_{\gamma}\to \kg_{\gamma}^{0}$ be the corresponding moment map, here $\kg_{\gamma}^{0}$ is the trace $0$ part of $\kg_{\gamma}:=\lie(G_{\gamma})$. Let $\co$ be the $G_{\gamma}$-orbit of a \emph{generic} (to be explained below) regular semisimple element in $\kg_{\gamma}^{0}$. The group $PG_{\gamma}$ acts freely on $\mu^{-1}(\co)$ and we have the geometric quotient $\mu^{-1}(\co)/PG_{\gamma}$, which is a smooth quasi-projective algebraic variety. Furthermore, the Weyl group $W_{\gamma}$ of $G_{\gamma}$ acts on the cohomological groups $H^{*}(\mu^{-1}(\co)/PG_{\gamma})$. One of the main results of \cite{positive} states that
$$
\Omega(\gamma)(q)=q^{\frac{1}{2}\chi_{Q}(\gamma,\gamma)}\sum_{i}\dim\big(H^{2i}(\mu^{-1}(\co)/PG_{\gamma})^{W_{\gamma}}\big)\, q^{i}.
$$

Based on this result, Hausel conjectured that the cohomological groups 
$$
H^{k}(\mu^{-1}(\co)/PG_{\gamma})^{W_{\gamma}}
$$ 
are geometric realisations of the generating set $V_{\gamma, k+\chi_{Q}(\gamma,\gamma)}^{\prim}$. (Of course, the conjecture is meaningful only when $\Omega(\gamma)(q)$ is non-zero. We will always impose this condition). In this article, we confirm this conjecture. Our construction goes as follows:

Let $\chi:\kg_{\gamma}^{0}\to \kg_{\gamma}^{0}\sslash G_{\gamma}\cong\kt_{\gamma}^{0}\sslash W_{\gamma}$ be the characteristic morphism. Following Ginzburg \cite{ginzburg}, we consider the composition $f:M_{\gamma}\xrightarrow{\mu} \kg_{\gamma}^{0}\xrightarrow{\chi} \kt_{\gamma}^{0}\sslash W_{\gamma}$. Let
$$\kt_{\gamma}^{0,\gen}:=\kt_{\gamma}^{0,\reg}\backslash \bigcup_{\substack{\gamma_{1},\gamma_{2}\in \bz^{I}_{\geq 0},\\ \gamma_{1}+\gamma_{2}=\gamma}}W_{\gamma}\cdot(\kt_{\gamma_{1}}^{0}\oplus\kt_{\gamma_{2}}^{0}),
$$ 
we call conjugates of elements in it \emph{generic} regular semisimple elements. Let 
$$
U_{\gamma}:=f^{-1}(\kt^{0, \gen}_{\gamma}/W_{\gamma}),
$$ 
then the group $PG_{\gamma}$ acts freely on $U_{\gamma}$ and the quotient $U_{\gamma}/PG_{\gamma}$ is a quasi-projective algebraic variety. Furthermore, the restriction of the morphism $f$ to $U_{\gamma}$ descends to a morphism $\bar{f}: U_{\gamma}/PG_{\gamma}\to \kt^{0,\gen}_{\gamma}/W_{\gamma}$. We'll prove that it makes $U_{\gamma}/PG_{\gamma}$ a fiber bundle on $\kt^{0,\gen}_{\gamma}/W_{\gamma}$ with fibers isomorphic to $\mu^{-1}(\co)/PG_{\gamma}$. Our main result is the following:

\begin{thm}\label{main}

The pure part of $H^{*}(U_{\gamma}/PG_{\gamma})$ is equal to $H^{*}(\mu^{-1}(\co)/PG_{\gamma})^{W_{\gamma}}$. The restriction $H_{PG_{\gamma}}^{*}(M_{\gamma})\to H^{*}(U_{\gamma}/PG_{\gamma})$ factors through and is surjective onto the pure part of the latter, and its restriction to $V_{\gamma,*+\chi_{Q}(\gamma,\gamma)}^{\prim}$ induces an isomorphism
$$
V_{\gamma,*+\chi_{Q}(\gamma,\gamma)}^{\prim}\cong H^{*}(\mu^{-1}(\co)/PG_{\gamma})^{W_{\gamma}}.
$$

\end{thm}

\section{Proof of the main theorem}

We begin by recalling the construction of Kontsevich and Soibelman of the cohomological Hall algebra. Given two vectors $\gamma_{1}, \,\gamma_{2}\in \bz_{\geq 0}^{I}$, let $\gamma=\gamma_{1}+\gamma_{2}$. The product $\chh_{\gamma_{1}}\times \chh_{\gamma_{2}}\to \chh_{\gamma}$ is defined to be the composition of the K\"unneth isomorphism
$$
H_{G_{\gamma_{1}}}^{*}(M_{\gamma_{1}})\otimes H_{G_{\gamma_{2}}}^{*}(M_{\gamma_{2}})\cong H^{*}_{G_{\gamma_{1}}\times G_{\gamma_{2}}}(M_{\gamma_{1}}\times M_{\gamma_{2}}),
$$
and the following morphisms
\begin{equation}\label{defcoha}
H^{*}_{G_{\gamma_{1}}\times G_{\gamma_{2}}}(M_{\gamma_{1}}\times M_{\gamma_{2}})\cong 
H^{*}_{G_{\gamma_{1},\gamma_{2}}}(M_{\gamma_{1},\gamma_{2}})
\xrightarrow{\phi_{1}} H^{*+2c_{1}}_{G_{\gamma_{1},\gamma_{2}}}(M_{\gamma})
\xrightarrow{\phi_{2}} H^{*+2c_{1}+2c_{2}}_{G_{\gamma}}(M_{\gamma}),
\end{equation}
where $c_{1}=\dim_{\bc}M_{\gamma}-\dim_{\bc}M_{\gamma_{1},\gamma_{2}}$ and $c_{2}=\dim_{\bc}G_{\gamma_{1}, \gamma_{2}}-\dim_{\bc}G_{\gamma}$, and the first isomorphism is induced by the fibrations in affine spaces 
$$
M_{\gamma_{1},\gamma_{2}}\to M_{\gamma_{1}}\times M_{\gamma_{2}},\quad  G_{\gamma_{1},\gamma_{2}}\to G_{\gamma_{1}}\times G_{\gamma_{2}},
$$
and the other morphisms $\phi_{1},\phi_{2}$ are natural push forwards.

\begin{lem}\label{vanishing}

Under the restriction $H^{*}([M_{\gamma}/G_{\gamma}])\to H^{*}([M_{\gamma}^{st}/G_{\gamma}])$, the image of 
$$\bigoplus_{\substack{\gamma_{1},\gamma_{2}\in \bz_{\geq 0}^{I}\\ \gamma_{1}+\gamma_{2}=\gamma}} \chh_{\gamma_{1}}\times \chh_{\gamma_{2}}
$$ 
in $\chh_{\gamma}=H^{*}([M_{\gamma}/G_{\gamma}])$ goes to $0$.

\end{lem}

\begin{proof}

By the definition of Gysin map, the morphism $\phi_{1}$ in the composition (\ref{defcoha}) factorises as
\[
H^{*}([M_{\gamma_{1},\gamma_{2}}/G_{\gamma_{1},\gamma_{2}}])
\to H^{*+2c_{1}}_{[M_{\gamma_{1},\gamma_{2}}/G_{\gamma_{1},\gamma_{2}}]}([M_{\gamma}/G_{\gamma_{1},\gamma_{2}}]) \to  H^{*+2c_{1}}([M_{\gamma}/G_{\gamma_{1},\gamma_{2}}]).
\]
Using the long exact sequence
$$
\cdots \to H^{*}_{[M_{\gamma_{1},\gamma_{2}}/G_{\gamma_{1},\gamma_{2}}]}([M_{\gamma}/G_{\gamma_{1},\gamma_{2}}])\to H^{*}([M_{\gamma}/G_{\gamma_{1},\gamma_{2}}])\to H^{*}([M_{\gamma}/G_{\gamma_{1},\gamma_{2}}]\backslash [M_{\gamma_{1},\gamma_{2}}/G_{\gamma_{1},\gamma_{2}}])\to \cdots,
$$
we see that $\mathrm{Im}(\phi_{1})$ goes to $0$ when we restrict it to 
$$
H^{*+2c_{1}}([M_{\gamma}/G_{\gamma_{1},\gamma_{2}}]\backslash [M_{\gamma_{1},\gamma_{2}}/G_{\gamma_{1},\gamma_{2}}]).
$$
Since $[M_{\gamma}^{st}/G_{\gamma_{1},\gamma_{2}}]$ is contained in $[M_{\gamma}/G_{\gamma_{1},\gamma_{2}}]\backslash [M_{\gamma_{1},\gamma_{2}}/G_{\gamma_{1},\gamma_{2}}]$, $\mathrm{Im}(\phi_{1})$ vanishes when we restrict it further to  
$$
H^{*+2c_{1}}([M_{\gamma}^{st}/G_{\gamma_{1},\gamma_{2}}]).
$$
Now applying $\phi_{2}$, we see that $\mathrm{Im}(\phi_{2}\circ \phi_{1})$ vanishes when we restrict it to $H^{*+2c_{1}+2c_{2}}([M_{\gamma}^{st}/G_{\gamma}])$.

\end{proof}

We need some preliminary results before proceeding to the proof of the main theorem. Given an element $t=(t_{i})_{i\in I}\in \kt_{\gamma}^{0,\gen}$, let $\co$ be its orbit under the action of $PG_{\gamma}$ by conjugation. Recall that Crawley-Boevey \cite{cb} has identified the geometric quotient $\mu^{-1}(\co)/PG_{\gamma}$ with a quiver variety: Let $\widetilde{Q}$ be the quiver obtained from $Q$ by attaching to each vertex $i\in I$ a leg of length $\gamma_{i}-1$. More precisely, vertices of $\widetilde{Q}$ are labeled $[i,j],i\in I, j=0,\cdots, \gamma_{i}-1$, and we identify $[i, 0]$ with $i$. Besides the arrows in $Q$, the new arrows in $\widetilde{Q}$ are $[i,j]\rightleftharpoons[i,j+1]$ for each $i\in I, j=0,\cdots, \gamma_{i}-2$. The new dimension vector of $\widetilde{Q}$ is defined to be $\widetilde{\gamma}_{[i,j]}=\gamma_{i}-j$. Again, we have the moment map $\widetilde{\mu}: M^{\widetilde{Q}}_{\widetilde{\gamma}}\to \kg_{\widetilde{\gamma}}^{0}$, where $M^{\widetilde{Q}}_{\widetilde{\gamma}}$ is the space of representations of $\widetilde{Q}$ of dimension vector $\widetilde{\gamma}$. For each $i\in I$, let $t_{i,1},\cdots, t_{i, \gamma_{i}}$ be the eigenvalues of $t_{i}$. Define $\lambda=(\lambda_{[i,j]})\in \kg_{\widetilde{\gamma}}^{0}$ to be 
\begin{eqnarray*}
\lambda_{[i,0]}&=&-t_{i,1},\\
\lambda_{[i,j]}&=& t_{i,j}-t_{i,j+1},\quad j=1,\cdots, \gamma_{i}-1.
\end{eqnarray*}
Notice that $\widetilde{\gamma}\cdot \lambda=0$. Now the result of Crawley-Boevey \cite{cb} states that 
\begin{equation}\label{cbiso}
\mu^{-1}(\co)/PG_{\gamma}\cong \widetilde{\mu}^{-1}(\lambda)/PG_{\widetilde{\gamma}}.
\end{equation}
Moreover, according to \cite{positive}, corollary 1.6 (iv), $\Omega(\gamma)(q)$ is non-zero if and only if $\widetilde{\gamma}$ is a positive root of $Q'$, here we write $\widetilde{Q}$ as the double of another quiver $Q'$.

\begin{lem}\label{fiberbundle}

The morphism $\bar{f}: U_{\gamma}/PG_{\gamma}\to \kt_{\gamma}^{0,\gen}/W_{\gamma}$ makes $U_{\gamma}/PG_{\gamma}$ a fiber bundle over $\kt_{\gamma}^{0,\gen}/W_{\gamma}$. Moreover, the sheaf $R^{i}\bar{f}_{*}\bq$ is constant on the étale neighbourhood $\kt^{0,\gen}_{\gamma}\to \kt^{0,\gen}_{\gamma}/W_{\gamma}$ of $\kt^{0,\gen}_{\gamma}/W_{\gamma}$.

\end{lem}

\begin{proof}

This is basically the lemma 48 of \cite{maffei}, with one difference. As in the proof of \cite{positive} theorem 2.3, Maffei works with quiver without loops, but his proof carries over in our case. In his proof, the important point is the surjectivity of $\bar{f}$ (or rather the hyperk\"ahler moment map on the generic locus, but this can be reduced to $\bar{f}$ by hyperk\"ahler rotation). According to \cite{cb1}, theorem 4.4, this is fulfilled in our situation since $
\mu^{-1}(\co)/PG_{\gamma}\cong \widetilde{\mu}^{-1}(\lambda)/PG_{\widetilde{\gamma}}
$ and $\widetilde{\gamma}\cdot \lambda =0$, taking into account that $\widetilde{\gamma}$ is a positive root of $Q'$.

\end{proof}

\begin{lem}[Crawley-Boevey-van den Bergh]\label{pure}
The smooth quasi-projective algebraic variety $\mu^{-1}(\co)/PG_{\gamma}$ has pure mixed Hodge structure.

\end{lem}

\begin{proof}

This is a corollary of \cite{cbb}, \S 2.4, taking into account the isomorphism (\ref{cbiso}). 

\end{proof}

Now we can prove the first part of theorem \ref{main}.

\begin{thm}\label{pure part}

The pure part of $H^{*}(U_{\gamma}/PG_{\gamma})$ is equal to $H^{*}(\mu^{-1}(\co)/PG_{\gamma})^{W_{\gamma}}$. 

\end{thm}

\begin{proof}

Consider the fiber bundle $\bar{f}: U_{\gamma}/PG_{\gamma}\to \kt^{0,\gen}_{\gamma}/W_{\gamma}$. By lemma \ref{fiberbundle}, the sheaf $R^{i}\bar{f}_{*}\bq$ is constant on the étale neighbourhood $\kt^{0,\gen}_{\gamma}\to \kt^{0,\gen}_{\gamma}/W_{\gamma}$ of $\kt^{0,\gen}_{\gamma}/W_{\gamma}$, so we get the Hochschild-Serre spectral sequence, 
$$
E_{2}^{p,q}=H^{p}(W_{\gamma}, H^{q}(\kt^{0,\gen}_{\gamma},\,R\bar{f}_{*}\bq))\Longrightarrow H^{p+q}(U_{\gamma}/PG_{\gamma}).
$$
Since $W_{\gamma}$ is a finite group, we have $E_{2}^{p,q}=0$ for $p\neq 0$. So the spectral sequence degenerates, and we get
\begin{eqnarray*}
H^{q}(U_{\gamma}/PG_{\gamma})&=&\big(H^{q}(\kt^{0,\gen}_{\gamma},\,R\bar{f}_{*}\bq)\big)^{W_{\gamma}}\\
&=&\Big(\bigoplus_{q_{1}+q_{2}=q}H^{q_{1}}(\kt_{\gamma}^{0,\gen})\otimes H^{q_{2}}(\mu^{-1}(\co)/PG_{\gamma})\Big)^{W_{\gamma}}.
\end{eqnarray*}

Since $\kt_{\gamma}^{0,\gen}$ is the complement of unions of sufficiently many hyperplanes in the vector space $\kt_{\gamma}^{0}$, one proves easily by induction on the number of hyperplanes that the mixed Hodge structure of $H^{i}(\kt_{\gamma}^{0,\gen})$ is not pure if $i\neq 0$. So the pure part of $H^{q}(U_{\gamma}/PG_{\gamma})$ is exactly $H^{q}(\mu^{-1}(\co)/PG_{\gamma})^{W_{\gamma}}$.

\end{proof}

To prove the second result in the main theorem, we need some facts from algebraic stacks. We refer the reader to Olsson-Laszlo \cite{ol1} and Sun \cite{sun} for the proofs. Although they work over the finite fields, their results apply in our situation since the quiver varieties are in fact $\bz$-schemes. The moduli stack $[M_{\gamma}/PG_{\gamma}]$ is a smooth Artin stack over $\bc$ with dimension $d_{\gamma}=$ $\dim(M_{\gamma})-\dim(PG_{\gamma})$. It has dualizing complex $\bq(d_{\gamma})[2d_{\gamma}]$, and we have the Poincaré duality, which is a perfect non-degenerate bilinear pairing
$$ 
H^{i}([M_{\gamma}/PG_{\gamma}])\times H^{2d_{\gamma}-i}_{c}([M_{\gamma}/PG_{\gamma}])\to \bq(d_{\gamma}).
$$
Using the fibration $[M_{\gamma}/PG_{\gamma}]\to [\,\mathrm{pt}/PG_{\gamma}]$, we have $H^{i}([M_{\gamma}/PG_{\gamma}])=H^{i}_{PG_{\gamma}}(\mathrm{pt})$ is pure of weight $i$, the groups $H^{2d_{\gamma}-i}_{c}([M_{\gamma}/PG_{\gamma}])$ are all pure of weight $2d_{\gamma}-i$.

Furthermore, let $Z_{\gamma}=M_{\gamma}\backslash U_{\gamma}$, then $H^{i}_{c}([Z_{\gamma}/PG_{\gamma}])$ is of weight less than or equal to $i$. This is essentially \cite{weil2}. More precisely, as in \cite{bl}, let $\{E_{n}\to B_{n}\}_{n\in \bn}$ be an injective system of finite dimensional $n$-acyclic approximation to the universal $PG_{\gamma}$-torsor $E\to B$, then 
$$
H^{i}_{c}([Z_{\gamma}/PG_{\gamma}])=\lim_{n\to \infty} H^{i+2\dim(E_{n})}_{c}(Z_{\gamma}\times_{PG_{\gamma}}E_{n})(-\dim(E_{n})).
$$
Now it suffices to apply \cite{weil2} to the right hand side.

\begin{proof}[Proof of the second part of theorem \ref{main}]

We have the long exact sequence
\begin{equation}\label{standard 1}
\cdots\to H^{i-1}_{c}([Z_{\gamma}/PG_{\gamma}])\rightarrow H_{c}^{i}([U_{\gamma}/PG_{\gamma}]) \rightarrow H_{c}^{i}([M_{\gamma}/PG_{\gamma}]) \to H^{i}_{c}([Z_{\gamma}/PG_{\gamma}])\to \cdots.
\end{equation}
Since $H^{i-1}_{c}([Z_{\gamma}/PG_{\gamma}])$ is of weight less than or equal to $i-1$, the pure part of $H_{c}^{i}([U_{\gamma}/PG_{\gamma}])$ injects into $H_{c}^{i}([M_{\gamma}/PG_{\gamma}])$. Taking Poincaré duality, we have that $H^{2d_{\gamma}-i}([M_{\gamma}/PG_{\gamma}])$ maps onto the pure part of $H^{2d_{\gamma}-i}([U_{\gamma}/PG_{\gamma}])$. By theorem \ref{pure part}, we have surjective morphism
$$
H^{j}([M_{\gamma}/PG_{\gamma}])\twoheadrightarrow H^{j}(\mu^{-1}(\co)/PG_{\gamma})^{W_{\gamma}}, \quad \forall\, j.
$$
By the definition of $\kt_{\gamma}^{0,\gen}$, we find easily that $U_{\gamma}\subset M_{\gamma}^{st}$. So the above map factorize by
\begin{equation}\label{standard 2}
H^{j}([M_{\gamma}/PG_{\gamma}])\to H^{j}(M_{\gamma}^{st}/PG_{\gamma}) \twoheadrightarrow H^{j}(\mu^{-1}(\co)/PG_{\gamma})^{W_{\gamma}}. 
\end{equation}
By lemma \ref{vanishing}, the first arrow has the same image as its restriction to $V_{\gamma, j+\chi_{Q}(\gamma,\gamma)}^{\prim}$, so we get a surjective morphism
$$
V_{\gamma, j+\chi_{Q}(\gamma,\gamma)}^{\prim}\twoheadrightarrow H^{j}(\mu^{-1}(\co)/PG_{\gamma})^{W_{\gamma}}. 
$$
By the result of \cite{positive} recalled in the introduction, they have the same dimension, so they are isomorphic.

\end{proof}

Similar arguments can be used to show the following variant of the geometric construction.

\begin{thm}\label{stablegenerater}

The restriction
$$H^{*}([M_{\gamma}/PG_{\gamma}])\to H^{*}([M_{\gamma}^{st}/PG_{\gamma}])$$
induces an isomorphism 
$$
V_{\gamma, *+\chi_{Q}(\gamma,\gamma)}^{\prim}\cong \mathrm{P}H^{*}([M_{\gamma}^{st}/PG_{\gamma}])
$$ 
where $\mathrm{P}H^{*}([M_{\gamma}^{st}/PG_{\gamma}])$ is the pure part of $H^{*}([M_{\gamma}^{st}/PG_{\gamma}])$. 

\end{thm}

\begin{proof}
The proof is almost the same as that of the main theorem, we indicate only the differences. Using an exact sequence as (\ref{standard 1}), with the pair ($U_{\gamma}$, $Z_{\gamma}$) replaced by ($M_{\gamma}^{st}$, $M_{\gamma}\backslash M_{\gamma}^{st}$), we can show that the restriction $H^{*}([M_{\gamma}/PG_{\gamma}])\to H^{*}([M_{\gamma}^{st}/PG_{\gamma}])$ factors through and is surjective onto $\mathrm{P}H^{*}([M_{\gamma}^{st}/PG_{\gamma}])$.
Again by lemma \ref{vanishing}, we get the surjection
$$
V_{\gamma, *+\chi_{Q}(\gamma,\gamma)}^{\prim}\twoheadrightarrow \mathrm{P}H^{*}([M_{\gamma}^{st}/PG_{\gamma}]).
$$
Now observe that the second morphism in the factorisation (\ref{standard 2}) has the same image as that of 
$$
\mathrm{P}H^{j}([M_{\gamma}^{st}/PG_{\gamma}])\rightarrow H^{j}(\mu^{-1}(\co)/PG_{\gamma})^{W_{\gamma}},
$$
since the morphism preserves the weights of the cohomological groups and $H^{j}(\mu^{-1}(\co)/PG_{\gamma})^{W_{\gamma}}$ is pure of weight $j$. So the factorisation (\ref{standard 2}) becomes
\begin{equation}\label{standard 3}
V_{\gamma, j+\chi_{Q}(\gamma,\gamma)}^{\prim}\twoheadrightarrow \mathrm{P}H^{j}([M_{\gamma}^{st}/PG_{\gamma}])\twoheadrightarrow H^{j}(\mu^{-1}(\co)/PG_{\gamma})^{W_{\gamma}}.
\end{equation}
Now that the composition is an isomorphism by our main theorem, all the arrows in (\ref{standard 3}) are isomorphisms.  

\end{proof}

\section*{Acknowledgements} We are very grateful to Tam\'as Hausel for having explained to us his conjecture, and to Ben Davison for pointing out a bug in the proof. We also want to thank an anonymous referee for his careful readings and helpful suggestions.

\end{document}